\documentclass[11pt]{amsart}
\usepackage{macrosS20} 
\usepackage[margin=0.9in]{geometry} 
\usepackage{appendix}
\usepackage{ mathrsfs }
\usepackage{xcolor}

\newcommand{\Adm}{{\mathrm{Adm}}}
\newcommand{\dd}{{\mathrm{d}}}

\DeclareMathOperator{\Real}{Re}

\numberwithin{equation}{section}

\begin{document}

\title[Classical solutions for HJB equations]{On classical solutions and canonical transformations for Hamilton--Jacobi--Bellman equations}

\author[M. Bansil]{Mohit Bansil}
\address{Department of Mathematics, University of California, Los Angeles, CA 90095-1555, USA}
\email{mbansil@math.ucla.edu} 

\author[A.R. M\'esz\'aros]{Alp\'ar R. M\'esz\'aros}  
\date{\today}
\address{Department of Mathematical Sciences, University of Durham, Durham DH1 3LE, UK}
\email{alpar.r.meszaros@durham.ac.uk} 

\begin{abstract}
In this note we show how canonical transformations reveal hidden convexity properties for deterministic optimal control problems, which in turn result in global existence of $C^{1,1}_{loc}$ solutions to first order Hamilton--Jacobi--Bellman equations.
\end{abstract}

\maketitle

\section{Introduction}

For given data $H:\R^d\times\R^d\to\R$ and $G:\R^d\to\R$ and time horizon $T>0$ let us consider the following Cauchy problem associated with the Hamilton--Jacobi--Bellman (HJB) equation
\begin{equation}\label{eq:HJ}
\left\{
\begin{array}{ll}
	\partial_t u(t,x) + H(x,\pa_x u) = 0, & (t,x)\in (0,T)\times\R^d, \\ 
	u(T, x) = G(x), & x\in\R^d.
\end{array}
\right.
\end{equation}

For convenience we will assume that $H \in C^2(\R^d\times\R^d)$ and $G\in C^2(\R^d)$, and the second derivatives of these functions are uniformly bounded. Suppose that $H$ is convex in its second variable, so that this can be seen as the Legendre--Fenchel transform of a Lagrangian function, i.e. we have 
\begin{equation}\label{def:H}
H(x,p)=\sup_{v\in\R^d}\{p\cdot v - L(x,-v)\},
\end{equation}
for some $L\in C^2(\R^d\times\R^d)$ given. For convenience, we also assume that $L$ is convex in its second variable. In this case \eqref{eq:HJ} corresponds to a variational problem. Indeed, it is well-known that under suitable assumptions we have that the value function
\begin{equation}\label{eq:control}
u(t,x):=\inf_{\gamma:[t,T] \to \R^d : \gamma(t) = x} \int_t^T L(\gamma(s), \dot \gamma(s)) \dd s + G(\gamma(T))
\end{equation}
is the unique viscosity solution to \eqref{eq:HJ}. This solution is locally Lipschitz continuous and locally semi-concave with a linear modulus of continuity (cf. \cite[Theorem 7.4.12, Theorem 7.4.14]{Cannarsa04}).

\medskip

It is well-known, however, that  the unique viscosity solution $u$ {in general} develops singularities in finite time, even if $H$ and $G$ are smooth. Examples of finite time singularity formations can be constructed in a relatively straightforward way, see for instance \cite[Example 1.3.4, Example 6.3.5]{Cannarsa04} or \cite[Appendix B.4]{TwoAuthor2022}. More particularly, one can show that for instance for purely quadratic Hamiltonians $H(x,p)=\frac12|p|^{2}$ a singularity of $u$ must form in finite time, for any {\it non-convex} (and non-constant) $G$. This fact is documented in \cite[Theorem 3.1]{GraMes22b}.

{Fine properties of sets of singularities of viscosity solutions have been studied extensively in the literature. Probably the first results on this topic were obtained in \cite{CanSon:87}. For a non-exhaustive list of further works dealing with singularities of solutions we refer the reader to \cite{CanSon:89, AlbCan:99, AlbCan:99, Yu, CanMenSin, CanYu, CanFra:91, CanFra:14} and to the review paper \cite{CanChe}.} Singularity formation is equivalent to the non-uniqueness of optimal trajectories in the variational problem \eqref{eq:control} (cf. \cite{Cannarsa04}).

Therefore, if one is able to ensure uniqueness of optimizers in \eqref{eq:control}, this results in the differentiability of the value function, hence in the existence of {a unique} classical solution to \eqref{eq:HJ}. This is precisely the case if $L$ is jointly convex and $G$ is convex, when the dynamics in the control problem is linear. This implies that $u(t,\cdot)$ inherits the convexity and thus it becomes a $C^{1,1}_{loc}$ classical solution for arbitrary time horizon. This fact and related properties are classical and well documented in the literature, see for instance \cite[Corollary 7.2.12]{Cannarsa04} and \cite{BarEva, GoeRoc, Goe:05-1,Goe:05-2, Roc:70-1, Roc:70-2}.

\medskip

As evidenced by the aforementioned works, the global existence of classical solutions is more of an exception than the rule in the theory of HJB equations.
To the best of our knowledge, beyond the fully convex regime described above, there are no alternative sufficient conditions on the data $(H,G)$ (or $(L,G)$) which would result in global in time classical well-posedness theory for \eqref{eq:HJ} in the class $C^{1,1}_{loc}$. Both of the concrete situations of singularity formations mentioned above (\cite[Appendix B.4]{TwoAuthor2022} and \cite[Theorem 3.1]{GraMes22b}) show precisely how semi-convexity estimates fail in general in finite time, hence this is a particular source of the failure of the global in time well-posedness in the class $C^{1,1}_{loc}$.

\medskip
 
In this manuscript we show that a special class of linear {\it canonical transformations} can reveal new global well-posedness theories. Let us describe the philosophy behind our approach. Let $\alpha\in\R$ be given. Then the transformation
$$
\R^{d}\times\R^{d}\ni(x,p)\mapsto (x,p-\alpha x) 
$$
is a so-called canonical transformation on the phase space, which preserves the structure of Hamilton's equations. Such transformations are well-known in classical mechanics (cf. \cite{Arnold1989}). We will make use of the following definitions.
\begin{equation}\label{def:H_alpha}
H_\alpha(x,p):=H(x,p-\alpha x)
\end{equation}
and 
\begin{equation}\label{def:G_alpha}
G_{\alpha}(x):=G(x)+\frac{\alpha}{2}|x|^{2}.
\end{equation}

Because of the nature of this transformation, we can state the first result of the paper.
\begin{thm}\label{prop:trans_HJ}
Let $\alpha\in\R$. Then $u$ is a classical solution to \eqref{eq:HJ} with data $(H,G)$ in $(0,T)\times\R^d$, if and only if $u_\alpha:(0,T)\times\R^{d}$, defined as 
$$
u_\alpha(t,x) := u(t,x) + \frac{\alpha}{2}|x|^2,
$$ 
is a classical solution to \eqref{eq:HJ} on $(0,T)\times\R^d$ with data $(H_\alpha,G_\alpha).$
\end{thm}

This theorem has two immediate consequences. First, if we have a global well-posedness theory for \eqref{eq:HJ} in the class $C^{1,1}_{loc}$ with data $(H,G)$, we obtain a whole one-parameter family of global well-posedness theories in $C^{1,1}_{loc}$ with data $(H_{\alpha},G_{\alpha})_{\alpha\in\R}$. Second, if we are able to find one real number $\alpha\in\R$ such that \eqref{eq:HJ} is globally well-posed for the data $(H_{\alpha},G_{\alpha})$, then the original problem with data $(H,G)$ must also be globally well-posed.

It turns out that this second consequence will be the one revealing genuine new global in time well-posedness theories for \eqref{eq:HJ} in the class $C^{1,1}_{loc}$. Therefore, as our second main result (formulated in Theorem \ref{thm:HJE_wellposed} below), we have identified sufficient conditions on $(H,G)$ which imply that for some precise $\alpha\in\R$, the transformed data $(H_{\alpha},G_{\alpha})$ (or the corresponding $(L_{\alpha},G_{\alpha})$) fall into the well-known fully convex regime, and therefore this gives the global well-posedness of \eqref{eq:HJ} in $C^{1,1}_{loc}$ with the original data $(H,G)$. A direct corollary of our main results can be summarized as follows.

\begin{cor}\label{cor:main} Let $H:\R^d\times\R^d\to\R$ and $G:\R^d\to \R$ be $C^2$ functions with uniformly bounded second order derivatives. Suppose furthermore that $H(x,\cdot)$ is strongly convex, uniformly in $x$.

Then, we have the following.
\begin{enumerate}
\item There exist a constant $C>0$ depending only on $\norm{D^2 G}_{\infty}$, $\norm{D^2 H}_{\infty}$, and the lower bound on $\pa_{pp} H$ such that \eqref{eq:HJ} with data $(\tilde H,G)$, where 
$$\tilde H(x,p) := H(x,p) + \alpha x\cdot p,$$ 
is globally well-posed in the class $C^{1,1}_{\rm{loc}}([0,T]\times\R^d)$, for any $T>0$, whenever $\alpha > C$.
\item There exist a constant $C>0$ depending only on $\norm{D^2 G}_{\infty}$, $\norm{D^2 H}_{\infty}$, and the lower bound on $\pa_{pp} H$ such that \eqref{eq:HJ} with data $(\tilde H, G)$, where
$$
\tilde H(x,p) := H(x,p) - \alpha \frac{|x|^2}{2},
$$ 
is globally well-posed in the class $C^{1,1}_{\rm{loc}}([0,T]\times\R^d)$, for any $T>0$, whenever $\alpha > C$.
\end{enumerate} 
\end{cor}
	
\begin{rmk}
We see that suitably modifying the existing data of a Hamilton--Jacobi--Bellman equation can `convexify' the problem, and in turn this leads to a global in time classical well-posedness theory. Corollary \ref{cor:main} shows that this procedure can be done not only by {adding the term $(x,p)\mapsto-\alpha\frac{|x|^2}{2}$ to the Hamiltonian, but also by adding $(x,p)\mapsto \alpha x\cdot p$ to $H$}, for a suitably chosen $\alpha$.
\end{rmk}	

\begin{rmk}
The attentive reader will notice that we are only using `upper triangular' canonical transformations, i.e. transformations of the form $(x,p)\mapsto (x,p-\alpha x)$. The reason for this is that in order for a system of ODEs {(in the $(X_{s},P_{s})$ unknowns)} to be the characteristic equations of a HJB equation not only do they need to have a Hamiltonian system structure but the boundary conditions must be of a particular form. Specifically to preserve the structure of boundary condition $X_0 = x_0$ we need to have the transformation to be upper triangular. From here one could imagine taking a transform $(x,p)\mapsto (x,p-A x)$ for some {constant} matrix $A$. In order to preserve the structure of the condition ${P_T} = \nabla G(X_T)$ (specifically that the right-hand side is the gradient of a function) we need that $A$ is symmetric. The choice of $A = \alpha I$ is taken for simplicity and the same arguments should work with any symmetric matrix $A$.
\end{rmk}

\begin{rmk}
Although we illustrate this canonical transformation technique to obtain a classical {global} well-posedness theory for the HJB equation, the approach works exactly the same for any other {features} of HJB equations. For example, if one had a result on the structure of the shocks for a HJB equation with data $(H,G)$ (e.g. for instance, the points of non-differentiability lie on a smooth curve) then the same result would hold for the transformed data $(H_{\alpha},G_{\alpha})$ (in fact the non-differentiability points would be the exact same).
\end{rmk}

We finish this introduction with some concluding remarks.
\begin{itemize}
\item In this manuscript, for the simplicity of the exposition, we choose to consider only a simple class of linear canonical transformations of the form $\R^{d}\times\R^{d}\ni(x,p)\mapsto (x,p-\alpha x)$. However, our approach would potentially work for a class of more general nonlinear transformations of the form
$$
\R^{d}\times\R^{d}\ni(x,p)\mapsto (x,p-\nabla\varphi(x)),
$$
for suitable potential functions $\varphi:\R^{d}\to\R$. While these extensions would not pose conceptual difficulties in our analysis, they would introduce heavy technical computations. Thus, we decided not to pursue these in this manuscript.
\item Our approach works using only a Hamiltonian perspective, therefore, in particular by working purely with Hamiltonian systems, we believe that similar results could be proven in the case of Hamiltonians which are not necessarily convex in the momentum variable. Again, for simplicity of the exposition, we do not pursue this direction here.
\item Canonical transformations are well understood in the case of more general Hamiltonian systems on symplectic manifolds, as these are symplectomorphisms on the cotangent bundle (cf. \cite{Arnold1989}). Although we consider only a simple Euclidean setting here, we suspect that our ideas could be potentially useful to study solutions of Hamilton--Jacobi--Bellman equations in more general geometric frameworks as well. Such studies would fall outside of the scope of the present manuscript.
\item It turns out that the canonical transformations considered in this manuscript reveal new deep well-posedness theories in a particular infinite dimensional setting, namely for the master equation in Mean Field Games. These results are detailed in our companion paper \cite{BanMes:master}.
\end{itemize}

\medskip

The rest of the paper contains two short sections. In Section \ref{sec:lag}, for pedagogical reasons, we detail the role of the specific canonical transformations from the Lagrangian perspective. Section \ref{sec:ham} contains our main results, and this is written purely from the Hamiltonian perspective.

\section{Canonical transformations and classical solutions from the Lagrangian perspective}\label{sec:lag}

For $t\in (0,T)$, consider the functional $\mathcal F_t : C^1((t,T);\R^d)\to\R$ defined as  
\[
\mathcal F_t(\gamma):= \int_t^T L(\gamma(x), \dot \gamma(s)) \dd s + G(\gamma(T)).
\]

Furthermore, we define the set of admissible curves as 
$$
\Adm_{t,x}:=\{\gamma\in C^1((t,T);\R^d): \gamma(t)=x\}.
$$
Using this functional, one has
$$
u(t,x):= \inf_{\gamma\in \Adm_{t,x}}\mathcal{F}_t(\gamma).
$$
Differentiability of solutions to \eqref{eq:HJ} is deeply linked to the uniqueness of minimizers in the optimal control problem \eqref{eq:control} (see for instance the discussions in \cite[Example 1.3.4, Example 6.3.5]{Cannarsa04} or \cite[Appendix B.4]{TwoAuthor2022}).
In particular, it is well-known that the convexity of the functional $\gamma\mapsto\mathcal{F}(\gamma)$ would imply that $u(t,\cdot)$ is convex, which in turn would further implies that $u$ is a classical solution to \eqref{eq:HJ} in the class $C^{1,1}_{\rm{loc}}([0,T]\times\R^d)$ (see \cite[Theorem 7.4.13]{Cannarsa04}). The convexity of $\gamma\mapsto\mathcal{F}(\gamma)$ can be guaranteed by the joint convexity of $L$ and the convexity of $G$.

However since the endpoint in the optimization problem, i.e. $\gamma(t)=x$, was fixed we could have just as well considered the functional 
\[
{\mathcal F}_{t,\alpha}(\gamma):= \mathcal F_t(\gamma) + \frac{\alpha}{2} |x|^2 =  \mathcal F_t(\gamma) + \frac{\alpha}{2} {|\gamma(t)|}^2,
\] 
for any $\alpha\in\R$. In this case, one would simply have
\begin{align}\label{eq:value_transform}
u(t,x)+\frac{\alpha}{2} |x|^2:= \inf_{\gamma\in\Adm_{t,x}}{\mathcal{F}}_{t,\alpha}(\gamma).
\end{align}

Following a standard idea in classical mechanics we rewrite this new term as the integral of its time derivative and an initial term to get
\[
{\mathcal F}_{t,\alpha}(\gamma)
&=  \int_t^T L(\gamma(x), \dot \gamma(s)) - \frac{\dd}{\dd s}\( \frac{\alpha}{2} {|\gamma(s)|}^2\) \dd s + \frac{\alpha}{2} {|\gamma(T)|}^2 + G(\gamma(T)) \\
&= \int_t^T L(\gamma(x), \dot \gamma(s)) - \alpha \inn{\gamma(s)}{\dot \gamma(s)} \dd s  + G(\gamma(T))+ \frac{\alpha}{2} {|\gamma(T)|}^2
\]
Notice now that it is possible to not have convexity of 
$$\gamma\mapsto\mathcal F_t(\gamma),$$
but have convexity of 
$$\gamma\mapsto {\mathcal F}_{t,\alpha}(\gamma),$$
 for some $\alpha\in\R$, even though 
$$\inf_{\gamma\in\Adm_{t,x}} \mathcal F_t(\gamma)\ \ {\rm and}\ \  \inf_{\gamma\in\Adm_{t,x}} {\mathcal F}_{t,\alpha}(\gamma)$$ 
are the same problems, in that the optimal values differ by the constant $\frac{\alpha}{2}|x|^2$ and in particular they have the same minimizers.

Therefore, it turns out that such transformations could reveal hidden convexity structures on the data, which were not straightforward in the original setting of the problem, and in particular \eqref{eq:HJ} would be well-posed in the class $C^{1,1}_{\rm{loc}}([0,T]\times\R^d)$, if $\gamma\mapsto{\mathcal F}_{t,\alpha}$ is convex even if we did not have convexity of $\gamma\mapsto{\mathcal F}_t$.

We also have the opposite situation, i.e. when $\gamma\mapsto\mathcal{F}_t$ is convex, yet $\gamma\mapsto{\mathcal{F}}_{t,\alpha}$ is not convex. As the convexity properties of the functionals can be characterized by the convexity of the Lagrangian and final data, it is natural to define the following quantities. For $\alpha\in\R$, let $L_\alpha:\R^d\times\R^d\to\R$ be defined as 
$$L_\alpha(x,v):=L(x,v) - \alpha x\cdot v$$ 
and $G_\alpha:\R^d\to\R$, defined as 
\begin{equation}\label{def:G_alpha}
G_\alpha(x):=G(x)+\frac{\alpha}{2}|x|^2.
\end{equation}

Based on the previous discussion, we can formulate the following proposition.

\begin{prop} We have the following.
\begin{enumerate}
\item[(i)] Let $G:\R^d\to\R$ be convex and let $L:\R^d\times\R^d\to\R$ be jointly convex, and suppose furthermore that both $G$ and $L$ have bounded second derivatives. Then, there exists $\alpha_0\in\R$ such that $L_\alpha$ is not jointly convex and $G_\alpha$ is not convex for any $\alpha<\alpha_0$.
\item[(ii)] There exist $L:\R^d\times\R^d\to\R$ not jointly convex and $G:\R^d\to\R$ non-convex, such that there for a suitable $\alpha\in\R$, $L_\alpha$ becomes jointly convex and $G_\alpha$ becomes convex.
\end{enumerate}
\end{prop}

\begin{proof} 
(i) Let $f:\R^d\times\R^d\to \R$ defined as $f(v,x) = v\cdot x$. Direct computation yields that $-1$ is an eigenvalue of $D^2f$ (the eigenvector is the $(1, \dots, 1, -1, \dots, -1)$). Indeed, we see that $D^{2}f(x,y) = \left[
\begin{array}{cc}
0_{d} & I_{d}\\
I_{d} & 0_{d}
\end{array}
\right],$
where $0_{d}$ and $I_{d}$ stand for the zero matrix and the identity matrix in $\R^{d\times d}$, respectively. From here the claim follows.
Hence for $\alpha < \hat{\alpha_0} := -\norm{D^2 L}_{\infty}$ we have that $L_\alpha$ is not jointly convex. Now let $\alpha_0:=\min\{\hat\alpha_0,-\|\partial_{xx}G\|_{L^\infty}\}$, and then the result follows.

\medskip

(ii) In the previous point we have constructed $L_\alpha$ and $G_\alpha$ that are non-convex, but $L$ and $G$ were convex. Now if we apply the same transformation on these new functions with constant $-\alpha$, i.e. $(L_{\alpha})_{-\alpha}$ and $(G_{\alpha})_{-\alpha}$ we get back to the original functions which were convex. The statement follows. 
\end{proof}

The transformations on $L$, as describe above, translate naturally to the Hamiltonian $H$. Indeed, we can see that $H_\alpha$ corresponding to $L_\alpha$ is defined as in \ref{def:H_alpha}.

It is important to notice that the previous transformation preserves the Hamiltonian structure and the HJB equation. This is what leads precisely to Theorem \ref{prop:trans_HJ}, whose proof is straightforward and we present it below.

\begin{proof}[Proof of Theorem \ref{prop:trans_HJ}]
This result readily follows from the representation formula \eqref{eq:value_transform}. Alternatively, direct computation yields
$$
\pa_t u_\alpha(t,x) = \pa_t u(t,x)\ \ {\rm{and}}\ \ \pa_x u_\alpha(t,x) = \pa_x u (t,x) + \alpha x,
$$
and so
$$
-\pa_t u_\alpha(t,x) + H_\alpha(x,\pa_x u_\alpha(t,x)) = -\pa_t u(t,x) + H(x,\pa_x u(t,x)+\alpha x - \alpha x) = 0,
$$
and
$$
u_\alpha(T,x) = G_\alpha(x).
$$
The result follows.
\end{proof}

\begin{rmk}
Because of the representation formula \eqref{eq:value_transform}, the previous result clearly holds true for viscosity solutions as well (cf. \cite[Theorem 7.4.14]{Cannarsa04}).
\end{rmk}

\section{Canonical transformations and classical solutions from the perspective of Hamiltonian systems}\label{sec:ham}

Based on \cite[Theorem 33.1]{Rockafellar1970}, we can formulate the following result.

\begin{lem}\label{lem:conc_conv0}
$H:\R^d\times\R^d\to\R$, defined in \eqref{def:H} is concave-convex (i.e. $H(\cdot,p)$ is convex for all $p\in\R^d$  and $H(x,\cdot)$ is convex for all $x\in\R^d$) if and only $L$ is jointly convex. 
\end{lem}

From this lemma we see that the global existence of classical solutions to the Hamilton--Jacobi equation \eqref{eq:HJ} in the class $C^{1,1}_{\rm{loc}}$, from the Hamiltonian point of view, is intimately linked to the concave-convex properties of $H$ and convexity of the final condition $G$.

\begin{defin}
For a square matrix $A\in\R^{m\times m}$, we define the symmetric matrix 
$$\Real A := \frac12(A+A^\top).$$  
For a symmetric matrix $A\in\R^{m\times m}$, we denote by $\lambda_{\min}(A)$ and $\lambda_{\max}(A)$ its smallest and largest eigenvalues, respectively. 
\end{defin}

\begin{lem}\label{lem:conc_conv}
Suppose that 
\begin{equation}\label{ass:disc}
\left(w^\top \Real \pa_{xp}H(x,p) w\right)^2 - \left(w^\top \pa_{pp} H(x,p) w\right) \left(w^\top \pa_{xx} H(x,p) w\right) \geq 0,\ \  \forall w \in \R^d,\ \ \forall x,p\in\R^d\times\R^d.
\end{equation} 
Define 
\small
\begin{equation}\label{def:eta}
\alpha := \inf_{{\footnotesize\begin{array}{c}(x,p, w)\in\R^{3d}\\ \norm{w} = 1\end{array}}} \frac{ w^\top \Real \pa_{xp}H(x,p) w + \sqrt{(w^\top \Real \pa_{xp}H(x,y) w)^2 - (w^\top \pa_{pp} H(x,p) w) (w^\top \pa_{xx} H(x,p) w)}} {w^\top \pa_{pp} H(x,p) w} 
\end{equation}
\normalsize
Suppose that $x\mapsto G(x) + \alpha\frac{\abs{x}^2}{2}$ is convex and
\small
\begin{equation}\label{ass:eta}
\alpha \geq \sup_{{\footnotesize\begin{array}{c}(x,p, w)\in\R^{3d}\\ \norm{w} = 1\end{array}}} \frac{ w^\top \Real \pa_{xp}H(x,p) w - \sqrt{(w^\top \Real \pa_{xp}H(x,p) w)^2 - (w^\top \pa_{pp} H(x,p) w) (w^\top \pa_{xx} H(x,p) w)}} {w^\top \pa_{pp} H(x,p) w}
\end{equation}
\normalsize
Then the Hamilton--Jacobi equation \eqref{eq:HJ} with data $(H,G)$ is globally well-posed in the class $C^{1,1}_{\rm{loc}}([0,T]\times\R^d)$. 
\end{lem}
\begin{proof}
Using \eqref{def:H_alpha} and \eqref{def:G_alpha} we define $H_\alpha$ and $G_\alpha$ with the particular choice of $\alpha$ given in the statement. We see that $G_{\alpha}$ is convex. Also, we compute for any $w \in \R^d$ and any $(x,p)\in\R^d\times\R^d$
\begin{align*}
w^\top \pa_{xx} H_\alpha(x,p) w &= w^\top \pa_{xx} H(x,p-\alpha x)w - 2\alpha w^\top \Real (\pa_{xp} H(x,p-\alpha x))w + \alpha^2 w^\top\pa_{pp} H(x,p-\alpha x) w
\end{align*}
This expression is a quadratic polynomial in $\alpha$ with positive leading coefficient. The conditions of the theorem assure that this polynomial is non-positive at $\alpha$, i.e.
\begin{align*}
&\frac{ w^\top \Real \pa_{xp}H(x,p) w - \sqrt{(w^\top \Real \pa_{xp}H(x,p) w)^2 - (w^\top \pa_{pp} H(x,p) w) (w^\top \pa_{xx} H(x,p) w)}} {w^\top \pa_{pp} H(x,p) w}\\
& \leq \alpha \\ 
&\leq  \frac{ w^\top \Real \pa_{xp}H(x,p) w + \sqrt{(w^\top \Real \pa_{xp}H(x,p) w)^2 - (w^\top \pa_{pp} H(x,p) w) (w^\top \pa_{xx} H(x,p) w)}} {w^\top \pa_{pp} H(x,p) w},
\end{align*}
for all $(x,p)\in\R^d\times\R^d$ and for all $w\in\R^d$.

In particular $H_\alpha$ is concave in $x$. Furthermore, this particular transformation does not change the convexity of $H_\alpha$ in the $p$-variable, as $\pa_{pp}H_\alpha(x,p) = \partial_{pp}H(x,p-\alpha x)$. The thesis of the lemma follows by Lemma \ref{lem:conc_conv0} and \cite[Theorem 7.4.13]{Cannarsa04}.
\end{proof}

As a consequence of this lemma, we can formulate the following result.

\begin{thm}\label{thm:HJE_wellposed}
We define the following quantities
$$\lambda_0:=\inf_{(x,p)\in\R^d\times\R^d}\lambda_{\min}\left(\Real \pa_{xp} H(x,p)\right),$$ 
$$\lambda_H:=\sup_{(x,p)\in\R^d\times\R^d}\lambda_{\max}\left(\pa_{xx} H(x,p)\right)$$
and
$$\lambda_G:=\inf_{x\in\R^d}\lambda_{\min}\left(\pa_{xx} G(x)\right).
$$ 
Suppose that 
\[\lambda_0^2 \geq \norm{\pa_{pp} H}_\infty \lambda_{H}\ \ {\rm{and\ \ that\ \ }}
\lambda_0 + \sqrt{\lambda_0^2 - \norm{\pa_{pp} H}_\infty \lambda_{H}} + \norm{\pa_{pp} H}_\infty \lambda_{G} \geq 0.
\]
Furthermore assume that either $\lambda_H \leq 0$ or $\lambda_0 \geq 0$. 
Then the Hamilton--Jacobi equation \eqref{eq:HJ} is globally well-posed, for any $T>0$, in the class $C^{1,1}_{\rm{loc}}([0,T]\times\R^d)$.
\end{thm}
\begin{proof}
We verify the assumptions of Lemma \ref{lem:conc_conv}. First, let us consider the inequality \eqref{ass:disc}.

If $\lambda_H \leq 0$ then \eqref{ass:disc} is fulfilled immediately. 

If $\lambda_H > 0$, but $\lambda_0 \geq 0$ we have the following. By definition of $\lambda_0$, $w^\top \Real \pa_{xp}H(x,p) w \geq \lambda_0$, for all $(x,p)\in\R^d\times\R^d$ and for all $w\in\R^d$. Since the right-hand side is non-negative we can square to obtain 
$$(w^\top \Real \pa_{xp}H(x,p) w)^2 \geq \lambda_0^2 \geq \norm{\pa_{pp} H}_\infty \lambda_{H},$$ 
which implies that 
$$(w^\top \Real \pa_{xp}H(x,p) w)^2 \ge (w^\top \pa_{pp} H(x,p) w) (w^\top \pa_{xx} H(x,p) w).$$ 
Thus, \eqref{ass:disc} follows.

\medskip

We verify the other assumptions in the statement of Lemma \ref{lem:conc_conv}. Let $\alpha$ be defined as in \eqref{def:eta}.
Just as before, we distinguish two cases.

{\it Case 1.} $\lambda_H \le 0$. 

In this case we have that $(w^\top \pa_{pp} H(x,p) w) (w^\top \pa_{xx} H(x,p) w)\le 0$, for all $(x,p)\in\R^d\times\R^d$ and for all $w\in\R^d$. Therefore
\[
\alpha \geq 0 \geq \frac{ w^\top \Real \pa_{xp}H(x,p) w - \sqrt{(w^\top \Real \pa_{xp}H(x,p) w)^2 - (w^\top \pa_{pp} H(x,p) w) (w^\top \pa_{xx} H(x,p) w)}} {w^\top \pa_{pp} H(x,p) w},
\] 
for all $(x,p)\in\R^d\times\R^d$ and for all $w\in\R^d$. This implies \eqref{ass:eta}.

Furthermore, we have
\small
\begin{align*}
\alpha &= \inf_{{\footnotesize\begin{array}{c}(x,p, w)\in\R^{3d}\\ \norm{w} = 1\end{array}}} \frac{ w^\top \Real \pa_{xp}H(x,p) w + \sqrt{(w^\top \Real \pa_{xp}H(x,y) w)^2 - (w^\top \pa_{pp} H(x,p) w) (w^\top \pa_{xx} H(x,p) w)}} {w^\top \pa_{pp} H(x,p) w}\\
&\ge \inf_{{\footnotesize\begin{array}{c}(x,p, w)\in\R^{3d}\\ \norm{w} = 1\end{array}}} \frac{ w^\top \Real \pa_{xp}H(x,p) w + \sqrt{(w^\top \Real \pa_{xp}H(x,y) w)^2 - \norm{\pa_{pp} H}_\infty (w^\top \pa_{xx} H(x,p) w)}} {\norm{\pa_{pp} H}_\infty},
\end{align*}
\normalsize
where in the last inequality we have used that the function $f:\{(a,b,c):\ c \geq 0,\ b\le 0,\ a^2\geq bc\}\to\R$ defined as  
$f(a,b,c) = \frac{a + \sqrt{a^2-bc}}{c}$ is decreasing in $c$. 

Continuing we have
\[
\alpha 
&\geq \inf_{{\footnotesize\begin{array}{c}(x,p, w)\in\R^{3d}\\ \norm{w} = 1\end{array}}} \frac{ w^\top \Real \pa_{xp}H(x,p) w + \sqrt{(w^\top \Real \pa_{xp}H(x,y) w)^2 - \norm{\pa_{pp} H}_\infty (w^\top \pa_{xx} H(x,p) w)}} {\norm{\pa_{pp} H}_\infty} \\
&\geq \inf_{{\footnotesize\begin{array}{c}(x,p, w)\in\R^{3d}\\ \norm{w} = 1\end{array}}} \frac{ w^\top \Real \pa_{xp}H(x,p) w + \sqrt{(w^\top \Real \pa_{xp}H(x,y) w)^2 - \norm{\pa_{pp} H}_\infty \lambda_H}} {\norm{\pa_{pp} H}_\infty} \\
&\geq \frac{ \lambda_0 + \sqrt{\lambda_0^2 - \norm{\pa_{pp} H}_\infty \lambda_H}} {\norm{\pa_{pp} H}_\infty}
\]
where the last inequality is because the function $f:\{(a,b):\ b \leq 0\}\to\R$ defined as 
$f(a,b) = a + \sqrt{a^2 - b}$ is increasing in $a$. From this, by the assumptions of this theorem it follows that $x\mapsto G(x) + \alpha\frac{\abs{x}^2}{2}$ is convex.

\medskip

{\it Case 2.} $\lambda_0\ge 0$. We notice that without loss of generality, we may assume also that the inequality $\lambda_H \geq 0$ takes place.

We have
\small
\begin{align}\label{chain:1}
\alpha &= \inf_{{\footnotesize\begin{array}{c}(x,p, w)\in\R^{3d}\\ \norm{w} = 1\end{array}}} \frac{ w^\top \Real \pa_{xp}H(x,p) w + \sqrt{(w^\top \Real \pa_{xp}H(x,y) w)^2 - (w^\top \pa_{pp} H(x,p) w) (w^\top \pa_{xx} H(x,p) w)}} {w^\top \pa_{pp} H(x,p) w} \\
\nonumber&\geq  \inf_{{\footnotesize\begin{array}{c}(x,p, w)\in\R^{3d}\\ \norm{w} = 1\end{array}}} \frac{ \lambda_0 + \sqrt{\lambda_0^2 - (w^\top \pa_{pp} H(x,p) w) \lambda_H}} {w^\top \pa_{pp} H(x,p) w} \\
\nonumber&\geq  \inf_{{\footnotesize\begin{array}{c}(x,p, w)\in\R^{3d}\\ \norm{w} = 1\end{array}}}\frac{ \lambda_0 + \sqrt{\lambda_0^2 - \norm{\pa_{pp} H}_\infty \lambda_H}} {w^\top \pa_{pp} H(x,p) w} \\
\nonumber&\geq \frac{ \lambda_0 + \sqrt{\lambda_0^2 - \norm{\pa_{pp} H}_\infty \lambda_H}} {\norm{\pa_{pp} H}_\infty}
\end{align}
\normalsize
where the last two inequalities follow from $\lambda_H \geq 0$ and $\lambda_0 \geq 0$ respectively. We notice also that in the previous chain of inequalities all the quantities under the square root are non-negative.

Furthermore, we see that $\lambda_0 + \sqrt{\lambda_0^2 - \norm{\pa_{pp} H}_\infty \lambda_{H}} + \norm{\pa_{pp} H}_\infty \lambda_{G} \geq 0$ implies that $\alpha + \lambda_G \geq 0$ and so $x\mapsto G(x) + \alpha\frac{\abs{x}^2}{2}$ is convex.  

Next note that the function $f:\{(a,b):\ a,b \geq 0,\ a^2 \geq b\}\to\R$ defined as 
$f(a,b) = a - \sqrt{a^2 - b}$ is decreasing in $a$. Hence
\begin{align}\label{chain:2}
&\frac{ w^\top \Real \pa_{xp}H(x,p) w - \sqrt{(w^\top \Real \pa_{xp}H(x,p) w)^2 - (w^\top \pa_{pp} H(x,p) w) (w^\top \pa_{xx} H(x,p) w)}} {w^\top \pa_{pp} H(x,p) w} \\
\nonumber&\leq \frac{ \lambda_0 - \sqrt{\lambda_0^2 - (w^\top \pa_{pp} H(x,p) w) (w^\top \pa_{xx} H(x,p) w)}} {w^\top \pa_{pp} H(x,p) w} \\
\nonumber&\leq \frac{ \lambda_0 - \sqrt{\lambda_0^2 - (w^\top \pa_{pp} H(x,p) w) \lambda_H}} {w^\top \pa_{pp} H(x,p) w} \\
\nonumber&\leq \frac{ \lambda_0 - \sqrt{\lambda_0^2 - \norm{\pa_{pp} H}_\infty \lambda_H}} {\norm{\pa_{pp} H}_\infty}
\end{align}
where the last inequality follows from the fact that the function $f:\{(a,b,c):\ a,b,c\geq 0,\ a^2\geq bc\}\to\R$, defined as 
$f(a,b,c) = \frac{a - \sqrt{a^2-bc}}{c}$ is increasing in $c$. 
Combining \eqref{chain:1} and \eqref{chain:2} we can conclude that
\small
\[
\alpha\ge \sup_{{\footnotesize\begin{array}{c}(x,p, w)\in\R^{3d}\\ \norm{w} = 1\end{array}}} \frac{ w^\top \Real \pa_{xp}H(x,p) w - \sqrt{(w^\top \Real \pa_{xp}H(x,y) w)^2 - (w^\top \pa_{pp} H(x,p) w) (w^\top \pa_{xx} H(x,p) w)}} {w^\top \pa_{pp} H(x,p) w},
\]
\normalsize
which completes the proof in this case. 

\end{proof}

From this theorem the proof of Corollary \ref{cor:main} is immediate.
\begin{proof}[Proof of Corollary \ref{cor:main}]
	We see that if $\alpha$ is large enough then the assumptions of Theorem \ref{thm:HJE_wellposed} are satisfied. 
	\end{proof}

{\bf Acknowledgements.} The authors are grateful to Wilfrid Gangbo for valuable remarks and constructive comments. MB's work was supported by the National Science Foundation Graduate Research Fellowship under Grant No. DGE-1650604 and by the Air Force Office of Scientific Research under Award No. FA9550-18-1-0502. ARM has been partially supported by the EPSRC New Investigator Award ``Mean Field Games and Master equations'' under award no. EP/X020320/1 and by the King Abdullah University of Science and Technology Research Funding (KRF) under award no. ORA-2021-CRG10-4674.2. Both authors acknowledge the partial support of the Heilbronn Institute
for Mathematical Research and the UKRI/EPSRC Additional Funding Programme for Mathematical Sciences through the focused research grant ``The master equation in Mean Field Games''.

\medskip
\medskip

{\bf Data Availability.} We did not make use of or did not generate data sets.

\medskip
\medskip

{\bf Conflict of interest.} The authors declare that they do not have any conflict of interest.

\bibliographystyle{alpha}
\bibliography{MeanField}

\end{document}